\documentclass[11pt]{amsart}

\usepackage{hyperref}

\usepackage{tikz}

\usetikzlibrary{knots}
\usetikzlibrary{decorations.markings}

\usepackage{graphicx} 
\usepackage{t1enc}
\usepackage{amsthm,amssymb}
\usepackage{latexsym}
\usepackage{amsmath} 
\usepackage{graphics}
\usepackage{epsfig,multicol}
\usepackage{amsfonts}
\usepackage[latin1]{inputenc}
\usepackage[english]{babel}
\usepackage{ mathrsfs }

\newtheorem{theorem}{Theorem}
\newtheorem*{theorem*}{Theorem}
\newtheorem*{proposition*}{Proposition}
\newtheorem*{lemma*}{Lemma}

\newtheorem{definition}{Definition}
\newtheorem*{notation}{Notation}
\newtheorem*{conjecture}{Conjectures}
\newtheorem{remark}{Remark}

{\theoremstyle{definition} }

\def\hpic #1 #2 {\mbox{$\begin{array}[c]{l} \epsfig{file=#1,height=#2}
\end{array}$}}
 
\def\vpic #1 #2 {\mbox{$\begin{array}[c]{l} \epsfig{file=#1,width=#2}
\end{array}$}}

\newcommand  {\rmn}\romannumeral

\newcommand {\CL}{\mathcal{L}}
\newcommand {\IN}{\mathbb{N}}

\newcommand {\CP}{\mathcal{P}}

\begin{document}
\title[]{A computational study of the number of connected components of positive Thompson links}
\author{Valeriano Aiello} 
\author{Stefano Iovieno}

\begin{abstract}
Almost a decade ago Vaughan Jones introduced a method to produce knots from elements of the Thompson groups $F$, 
which was later extended to the Brown-Thompson group $F_3$. 
In this article we define a way to produce permutations out of elements of the $F$ and $F_3$ that we call Thompson permutations. 
The number of orbits of each Thompson permutation coincides with the number of connected components of the link.
We explore the positive elements of $F_3$ of fixed \emph{width} and \emph{height} and make some conjectures based on numerical experiments.
In order to define the Thompson permutations we need to assign an orientation to each   link produced from elements of $F$ and $F_3$.
We prove that all oriented links can be produced in this way.
\end{abstract}
\maketitle
 
 \section*{Introduction} 
 The Thompson group $F$ is an interesting group that was introduced by Richard Thompson in the sixties and so far it has attracted attention in several areas of
 mathematics, such as logic, topology, dynamics, geometric group theory, see e.g. \cite{CFP}. It admits many equivalent definitions, in terms of homeomorphisms of the unit interval, tree diagrams, strand diagrams, 
 as a diagram group, to name but a few \cite{CFP, BM, GS}.
In this article, $F$  appears because of Jones's reconstruction program of conformal field theories.
Within this  framework, Vaughan Jones defined a method to construct actions of the Thompson and "Thompson" like groups, \cite{Jo16}. 
In a special application of this method, it surfaced Jones's construction of
knots and links, which was first defined for the Thompson group $F$.
Later, by seeing $F$ as a subgroup of the Brown-Thompson group $F_3$, it was extended  to the Brown-Thompson group $F_3$. 
Given an element $g$ in $F$ or $F_3$, we denote by $\CL(g)$ the corresponding link (the exact construction will be recalled  in Section \ref{sec1}).
This procedure is analogous to the operation of closure on the braid groups, 
where knots and links can be obtained by \emph{closing} braids. 
For this reason we call \emph{closure} operation of producing a link from an element of $F$ or $F_3$.
The Thompson and the Brown-Thompson groups turn out to be as good as the braid groups at producing (unoriented) links. 
In fact, Jones proved an Alexander-type theorem, \cite{Jo14, Jo18}, that is, he proved that every  link can be obtained as the closure of a suitable element of $F$.
Since these links did not possess a natural orientation, Jones introduced the oriented versions of these groups: the oriented subgroups $\vec{F}$
and $\vec{F}_3$. An Alexander-type theorem holds in this context as well, see \cite{A}.

Along with this knot construction, several unitary representations came along \cite{Jo14, Jo16, Jo19, ACJ, ABC, AJ, BJ, AiCo1, AiCo2, BJ, TV, TV2, TV3, Bro}. 
These representations were defined either by means of planar algebras \cite{jo2} or Pythagorean/Cuntz algebras \cite{BJ, Cuntz}.
Thanks to this, progress was made on the study of infinite index maximal subgroups of the Thompson groups, \cite{GS2, TV, TV2}.

Like the braid groups, the groups $F$ and $F_3$ possess natural monoids 
$F_+$ and $F_{3,+}$ whose elements are called positive.
The links produced by these monoids are called positive Thompson knots and positive Brown-Thompson knots. The former class of knots has been studied
in \cite{AB2} by S. Baader and the first named author of this paper, where it was shown that this class is contained in the family of arborescent knots in the sense of Conway.
The oriented subgroup $\vec{F}$ also contains a monoid $\vec{F}_{+}$ of positive elements ($\vec{F}_{+}\subset F_+$), whose knots are positive (in the oriented sense),  see \cite{AB}.

 This work is motivated by \cite[Question 4]{Jo18}, where Jones asked to provide a group theoretical interpretation of the number of components of $\CL(g)$.
The analogy with the braid groups has so far been used as a guide  in this project. Recall that the presentation of braid group $B_n$ 
and that of the symmetric group $S_n$ differ only by one additional relation stating that the generators are involutive.
There is therefore a natural projection from $B_n$ to $S_n$.
It is well known that the number of orbits of such permutations (which is the same thing as the number of cycles in the cycle decomposition of the permutation) coincide with the number of connected components of the corresponding links.
Inspired by this, we define a method to produce permutations out of elements of $F$ and $F_3$. We call these permutations Thompson permutations and, by construction, the number of orbits of these permutations coincide with the number of connected components of the link produced with Jones's construction of knots.
The definition of these permutations requires the choice of an orientation for the link. We thus fix a convention to orient any link produced from $F$ and $F_3$, and prove an Alexander type theorem (Theorem \ref{theo1}), that is, for every oriented link $\vec{L}$, there exist an element of $F$ whose corresponding oriented link is equal to $\vec{L}$.
This means that for producing oriented links from Thompson groups,  two options are available: 1) use the oriented subgroups $\vec{F}$ and $\vec{F}_3$, \cite{Jo14, Jo18, A}, 2) use the Thompson groups $F$ and $F_3$.

In the last section, we summarize some computer assisted explorations of 1350210 positive elements, 
 and formulate some conjectures. 
More precisely, we fix the \emph{width}, while varying the \emph{height} of positive elements and group the corresponding permutations with respect to the number of orbits, we call these groups \emph{classes}. 
We then make conjectures on the maximum number of orbits 
and 
 the largest classes.

 \section{Preliminaries and notations.}\label{sec1}
  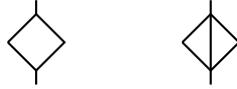
\begin{figure}
 \[\begin{tikzpicture}[x=.75cm, y=.75cm,
    every edge/.style={
        draw,
      postaction={decorate,
                    decoration={markings}
                   }
        }
]

\draw[thick] (0,0)--(.5,.5)--(1,0)--(.5,-.5)--(0,0);
\draw[thick] (0.5,0.75)--(.5,.5);
\draw[thick] (0.5,-0.75)--(.5,-.5);
\end{tikzpicture}
\qquad 
\qquad 
\begin{tikzpicture}[x=.75cm, y=.75cm,
    every edge/.style={
        draw,
      postaction={decorate,
                    decoration={markings}
                   }
        }
]

\draw[thick] (0.5,0.75)--(.5,-.75);
\draw[thick] (0,0)--(.5,.5)--(1,0)--(.5,-.5)--(0,0);
\end{tikzpicture}
\]
\caption{Pairs of opposing carets in $F$ and $F_3$.}\label{AAA}
\end{figure}
 \begin{figure}
\[\begin{tikzpicture}[x=.75cm, y=.75cm,
    every edge/.style={
        draw,
      postaction={decorate,
                    decoration={markings}
                   }
        }
]

\draw[thick] (0,0)--(.5,.5)--(1,0);
\draw[thick] (0.5,0.75)--(.5,.5);
 \node at (1.75,0.25) {$\scalebox{1}{$\mapsto$}$};

\end{tikzpicture}\, \,
\begin{tikzpicture}[x=.75cm, y=.75cm,
    every edge/.style={
        draw,
      postaction={decorate,
                    decoration={markings}
                   }
        }
]

\draw[thick] (0.5,0.75)--(.5,0);
\draw[thick] (0,0)--(.5,.5)--(1,0);
 \end{tikzpicture}
\]
 \caption{The monomorphism $\iota: F\to F_3$.}\label{f2f3}
\end{figure}
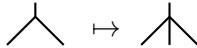

 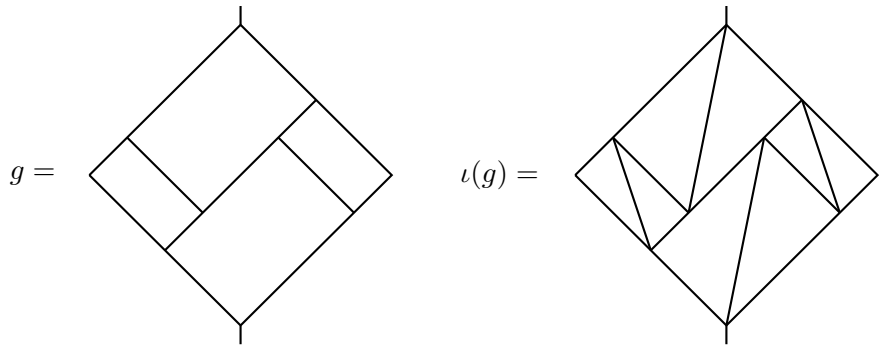
\begin{figure}[h]
\[
\begin{tikzpicture}[x=1cm, y=1cm,
    every edge/.style={
        draw,
      postaction={decorate,
                    decoration={markings}
                   }
        }
]

\node at (-.75,0) {$\scalebox{1}{$g=$}$};

\draw[thick] (0,0)--(2,2)--(4,0)--(2,-2)--(0,0);
\draw[thick]  (3,1)--(1,-1);
\draw[thick]  (2.5,.5)--(3.5,-.5);
\draw[thick]  (.5,.5)--(1.5,-.5);
\draw[thick]  (2,2)--(2,2.25);
\draw[thick]  (2,-2)--(2,-2.25);

\end{tikzpicture}
\qquad 
\begin{tikzpicture}[x=1cm, y=1cm,
    every edge/.style={
        draw,
      postaction={decorate,
                    decoration={markings}
                   }
        }
]

\node at (-1,0) {$\scalebox{1}{$\iota(g)=$}$};

\draw[thick] (0,0)--(2,2)--(4,0)--(2,-2)--(0,0);
\draw[thick]  (3,1)--(1,-1);
\draw[thick]  (2.5,.5)--(3.5,-.5);
\draw[thick]  (.5,.5)--(1.5,-.5);
 
 \draw[thick]  (.5,.5)--(1,-1);
 \draw[thick]  (3,1)--(3.5,-.5);
 \draw[thick]  (2,2)--(1.5,-.5);
 \draw[thick]  (2,-2)--(2.5,.5);

\draw[thick]  (2,2)--(2,2.25);
\draw[thick]  (2,-2)--(2,-2.25);

\end{tikzpicture}
\]
\caption{A element of $F$ and its image under the injection $\iota: F\to F_3$.}\label{fig2}
\end{figure}

There are several equivalent definitions of the Thompson groups $F$ and of the Brown-Thompson group $F_3$.
In this section we review the definitions
that are most appropriate for our work in this paper, namely the one that use tree diagrams. 
For further information we refer to  \cite{CFP, B} and \cite{Brown}.

A binary tree diagram is a pair of rooted, planar, binary trees $(T_+,T_-)$ with the same number of leaves.
Similarly, a ternary tree diagram is a pair of rooted, planar, ternary trees $(T_+,T_-)$ with the same number of leaves.
A common notation for these pairs is $\frac{T_+}{T_-}$.
As usual, we draw a pair of trees in the $xy$-plane with one tree upside down on top of the other. 
More precisely, the leaves of the trees sit on the natural numbers of the $x$-axis.
In both cases, the former tree is called top tree, the other bottom tree.
Two pairs of  trees are equivalent if they differ by a pair of opposing carets, see   Figure \ref{AAA}. 
The elements of $F$ are given by an equivalence classes of pairs of binary tree diagrams.
While elements of $F_3$ are equivalence classes of ternary tree diagrams.
For both $F$ and $F_3$ and for each element in them, there exists a minimal representative, in the sense a tree diagrams whose top and bottom trees have minimal number of leaves.
These representative are called reduced.
Thanks to the equivalence relation, the following rule defines the multiplication in both  $F$ and $F_3$: $(T_+,T)\cdot (T,T_-):=(T_+,T_-)$. The trivial element is represented by any pair $(T,T)$  
and the inverse of $(T_+,T_-)$ is $(T_-,T_+)$.

In Figure \ref{f2f3} we describe a natural injection $\iota: F\hookrightarrow F_3$: 
Given $(T_+,T_-)\in F$ we turn every trivalent vertex into a $4$-valent 
and join the new edges in the only possible planar way.
See Figure  \ref{fig2} for an example.
 
 The
 Thompson group $F$ and
  Brown-Thompson group $F_3$ also admits an infinite presentation
\begin{align*}
F & = \langle y_0, y_1, \ldots \; | \; y_n y_l = y_l y_{n+1} \;\;  \forall \; l<n\rangle\; ,\\
F_3 &=\langle  x_0, x_1, \ldots \; | \; x_n x_l = x_l x_{n+2} \;\;  \forall \; l<n\rangle\; . 
\end{align*}
See Figures \ref{genF} and \ref{genF3} for the tree diagrams of the generators of $F$ and $F_3$.
Actually $\{y_0, y_1\}$ and $\{x_0, x_1, x_2\}$ suffice to generate $F$ and $F_3$, respectively \cite{B, Brown}.

 \begin{figure}
\[
\begin{tikzpicture}[x=.35cm, y=.35cm,
    every edge/.style={
        draw,
      postaction={decorate,
                    decoration={markings}
                   }
        }
]

\node at (-1.5,0) {$\scalebox{1}{$y_0=$}$};
\node at (-1.25,-4) {\;};

\draw[thick] (0,0) -- (2,2)--(4,0)--(2,-2)--(0,0);
 \draw[thick] (1,1) -- (2,0)--(3,-1);

 \draw[thick] (2,2)--(2,2.5);

 \draw[thick] (2,-2)--(2,-2.5);

\end{tikzpicture}
\;\;
\begin{tikzpicture}[x=.35cm, y=.35cm,
    every edge/.style={
        draw,
      postaction={decorate,
                    decoration={markings}
                   }
        }
]

\node at (-3.5,0) {$\scalebox{1}{$y_1=$}$};
\node at (-1.25,-4) {\;};

\draw[thick] (2,2)--(1,3)--(-2,0)--(1,-3)--(2,-2);

\draw[thick] (0,0) -- (2,2)--(4,0)--(2,-2)--(0,0);
 \draw[thick] (1,1) -- (2,0)--(3,-1);

 \draw[thick] (1,3)--(1,3.5);
 \draw[thick] (1,-3)--(1,-3.5);

\end{tikzpicture}
\;\;
\begin{tikzpicture}[x=.35cm, y=.35cm,
    every edge/.style={
        draw,
      postaction={decorate,
                    decoration={markings}
                   }
        }
]

\node at (-5.5,0) {$\scalebox{1}{$y_3=$}$};
\node at (6,0) {$\ldots$};

\draw[thick] (2,2)--(1,3)--(-2,0)--(1,-3)--(2,-2);
\draw[thick] (1,3)--(0,4)--(-4,0)--(0,-4)--(1,-3); 

\draw[thick] (0,0) -- (2,2)--(4,0)--(2,-2)--(0,0);
 \draw[thick] (1,1) -- (2,0)--(3,-1);

 \draw[thick] (0,4)--(0,4.5);
 \draw[thick] (0,-4)--(0,-4.5);

\end{tikzpicture}
\]
\caption{The generators of $F=F_2$.}\label{genF}
\phantom{This text will be invisible} 
\end{figure}
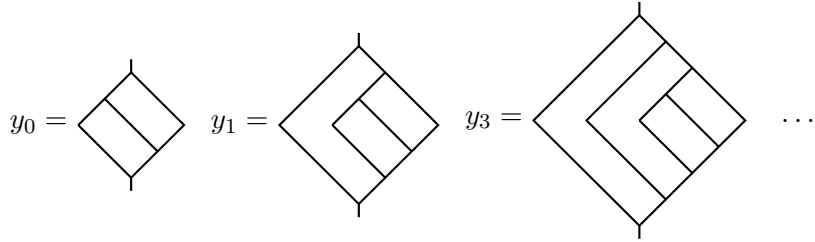
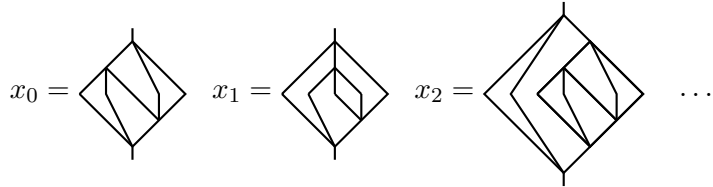
\begin{figure}
\[
\begin{tikzpicture}[x=.35cm, y=.35cm,
    every edge/.style={
        draw,
      postaction={decorate,
                    decoration={markings}
                   }
        }
]

\node at (-1.5,0) {$\scalebox{1}{$x_0=$}$};
\node at (-1.25,-3.25) {\;};

\draw[thick] (0,0) -- (2,2)--(4,0)--(2,-2)--(0,0);
\draw[thick] (1,1) -- (1,0)--(2,-2);
\draw[thick] (1,1) -- (2,0)--(3,-1);
\draw[thick] (2,2) -- (3,0)--(3,-1);

 \draw[thick] (2,2)--(2,2.5);
 \draw[thick] (2,-2)--(2,-2.5);

\end{tikzpicture}
\;\;
\begin{tikzpicture}[x=.35cm, y=.35cm,
    every edge/.style={
        draw,
      postaction={decorate,
                    decoration={markings}
                   }
        }
]

\node at (-1.5,0) {$\scalebox{1}{$x_1=$}$};
\node at (-1.25,-3.25) {\;};


\draw[thick] (0,0) -- (2,2)--(4,0)--(2,-2)--(0,0);
 \draw[thick] (1,0)--(2,-2);
\draw[thick] (3,0)--(2,1) -- (1,0); 
\draw[thick] (2,0)--(3,-1);
\draw[thick] (2,2) -- (2,0);
\draw[thick] (3,0)--(3,-1);


 \draw[thick] (2,2)--(2,2.5);
 \draw[thick] (2,-2)--(2,-2.5);
\end{tikzpicture}
\;\;
\begin{tikzpicture}[x=.35cm, y=.35cm,
    every edge/.style={
        draw,
      postaction={decorate,
                    decoration={markings}
                   }
        }
]

\node at (-3.5,0) {$\scalebox{1}{$x_2=$}$};
\node at (-1.25,-3.25) {\;};

\draw[thick] (2,2)--(1,3)--(-2,0)--(1,-3)--(2,-2);

\draw[thick] (0,0) -- (2,2)--(4,0)--(2,-2)--(0,0);
 \draw[thick] (1,1) -- (2,0)--(3,-1);

\draw[thick] (0,0) -- (2,2)--(4,0)--(2,-2)--(0,0);
\draw[thick] (1,1) -- (1,0)--(2,-2);
\draw[thick] (1,1) -- (2,0)--(3,-1);
\draw[thick] (2,2) -- (3,0)--(3,-1);

\draw[thick] (1,3) -- (-1,0)--(1,-3);

\node at (6,0) {$\ldots$};

 \draw[thick] (1,3)--(1,3.5);
 \draw[thick] (1,-3)--(1,-3.5);

\end{tikzpicture}
\]
\caption{The generators of $F_3$.}\label{genF3}
\phantom{This text will be invisible} \end{figure}

The positive monoid $F_+$ of $F$ is the monoid generated by $\{y_i\}_{i\geq 0}$ (but not their inverses) and
 consists of the binary tree diagrams whose bottom tree may be chosen of the following form
\[
\begin{tikzpicture}[x=.6cm, y=.6cm,
    every edge/.style={
        draw,
      postaction={decorate,
                    decoration={markings}
                   }
        }
]
\node (bbb) at (-2,-2) {$\scalebox{1}{$T_-=$}$}; 

\draw[thick] (0,0)--(4,-4)--(8,0);
\draw[thick] (4.5,-3.5)--(1,0);
\draw[thick] (5,-3)--(2,0);
\draw[thick] (7,-1)--(6,0);
\draw[thick] (7.5,-.5)--(7,0);
\draw[thick] (4,-4)--(4,-4.5);

\node (aaaa) at (5,-1) {$\scalebox{1}{$\ldots$}$}; 

\end{tikzpicture}\]
Similarly, the positive monoid $F_{3,+}$ of $F_3$
 is  generated by $\{x_i\}_{i\geq 0}$ and its elements
 admit ternary tree diagrams whose bottom tree is of the form 
\[
\begin{tikzpicture}[x=.7cm, y=.7cm,
    every edge/.style={
        draw,
      postaction={decorate,
                    decoration={markings}
                   }
        }
]

\draw[thick] (0,0)--(4,-4)--(8,0);
\draw[thick] (4.5,-3.5)--(1,0);
\draw[thick] (5,-3)--(2,0);
\draw[thick] (7,-1)--(6,0);
\draw[thick] (7.5,-.5)--(7,0);
\draw[thick] (4,-4)--(4,-4.5);

\draw[thick] (0.5,0)--(4,-4);
\draw[thick] (1.5,0)--(4.5,-3.5);
\draw[thick] (2.5,0)--(5,-3);
\draw[thick] (7,-1)--(6.5,0);
\draw[thick] (7.5,-.5)--(7.5,0);

\node (aaaa) at (5,-1) {$\scalebox{1}{$\ldots$}$}; 

\end{tikzpicture}
\]
Note that $\iota(y_i)=x_{2i}$ for all $i$. In particular,
$\iota(F_+)$ is contained in $F_{3,+}$.
 
 We now review Jones's construction of knots from elements of $F_3$ by giving an explicit example. We refer to \cite{Jo18} and \cite{A2} for survey articles on this area.
 Consider the element of $F_3$
 \[\begin{tikzpicture}[x=.75cm, y=.75cm,
    every edge/.style={
        draw,
      postaction={decorate,
                    decoration={markings}
                   }
        }
]

\draw[thick] (1,0)--(3,2)--(5,0);
\draw[thick] (3,1)--(3,2);
\draw[thick] (2,0)--(3,1);
\draw[thick] (3,0)--(3,1);
\draw[thick] (4,0)--(3,1);
\draw[thick] (3,2.5)--(3,2);

\draw[thick] (3,-2.5)--(3,-2);

\draw[thick] (1,0)--(3,-2)--(5,0);
\draw[thick] (2,0)--(3,-2);
\draw[thick] (3,0)--(4,-1);
\draw[thick] (4,0)--(4,-1);

\node at (-1,0) {$\scalebox{1}{$X=\frac{T_+}{T_-}=$}$};

\node at (0,-1.2) {\phantom{$\frac{T_+}{T_-}=$}};

\end{tikzpicture}
\]
Now join the two roots by an edge. Wolog we may suppose that the new edge passes through the point $(0,0)$.
\[\begin{tikzpicture}[x=.5cm, y=.5cm,
    every edge/.style={
        draw,
      postaction={decorate,
                    decoration={markings}
                   }
        }
]

\draw[thick] (1,0)--(3,2)--(5,0);
\draw[thick] (3,1)--(3,2);
\draw[thick] (2,0)--(3,1);
\draw[thick] (3,0)--(3,1);
\draw[thick] (4,0)--(3,1);
 \draw[thick] (-1,2) to[out=90,in=90] (3,2);

 
\draw[thick] (1,0)--(3,-2)--(5,0);
\draw[thick] (2,0)--(3,-2);
\draw[thick] (3,0)--(4,-1);
\draw[thick] (4,0)--(4,-1);

\draw[thick] (-1,-2) to[out=-90,in=-90] (3,-2);  
\draw[thick] (-1,-2)--(-1,2);

\node at (-1.5,0) {$\scalebox{1}{$\;$}$};

\node at (0,-1.2) {$\;$};

\end{tikzpicture}
\]
At this stage all the vertices are $4$-valent, change them according to the  rule displayed in Figure \ref{figrulescross} to obtain a knot diagram.
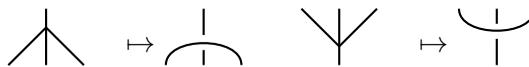
\begin{figure}[h] 
\[\begin{tikzpicture}[x=1cm, y=1cm,
    every edge/.style={
        draw,
      postaction={decorate,
                    decoration={markings}
                   }
        }
]


\draw[thick] (0.5,0.75)--(.5,0);
\draw[thick] (0,0)--(.5,.5)--(1,0);
 \node at (1.75,0.25) {$\scalebox{1}{$\mapsto$}$};

\end{tikzpicture}
\begin{tikzpicture}[x=1cm, y=1cm,
    every edge/.style={
        draw,
      postaction={decorate,
                    decoration={markings}
                   }
        }
]

 \draw[thick] (0.5,0.75)--(.5,.4);
\draw[thick] (0.5,0)--(.5,.2);

 \draw[thick] (0,0) to[out=90,in=90] (1,0);

\end{tikzpicture}
\qquad
\begin{tikzpicture}[x=1cm, y=1cm,
    every edge/.style={
        draw,
      postaction={decorate,
                    decoration={markings}
                   }
        }
]


\draw[thick] (0.5,0.75)--(.5,0);
\draw[thick] (0,0.75)--(.5,.25)--(1,.75);
 \node at (1.75,0.25) {$\scalebox{1}{$\mapsto$}$};

\end{tikzpicture}
\begin{tikzpicture}[x=1cm, y=1cm,
    every edge/.style={
        draw,
      postaction={decorate,
                    decoration={markings}
                   }
        }
]

 \draw[thick] (0.5,0.75)--(.5,.55);
\draw[thick] (0.5,0)--(.5,.35);

 \draw[thick] (0,.75) to[out=-90,in=-90] (1,.75);

\end{tikzpicture}
\]
\caption{The rules needed 
to turn $4$-valent vertices into crossings.}\label{figrulescross}
\phantom{This text will be invisible} 
\end{figure}

Therefore, in our example we get the knot $\CL(T_+,T_-)$
\[\begin{tikzpicture}[x=.5cm, y=.5cm,
    every edge/.style={
        draw,
      postaction={decorate,
                    decoration={markings}
                   }
        }
]

\draw[thick] (1,0)--(1,1);
 \draw[thick] (1,1) to[out=90,in=90] (5,1);  
 \draw[thick] (2,0) to[out=90,in=90] (4,0);  
\draw[thick] (3,0)--(3,.3);
\draw[thick] (3,.7)--(3,2);
\draw[thick] (5,0)--(5,1);

 \draw[thick] (-1,2) to[out=90,in=90] (3,2.35);

 
\draw[thick] (1,0)--(1,-.75);
\draw[thick] (2,0)--(2,-1.4);
\draw[thick] (3,0) to[out=-90,in=-90] (5,0);  
\draw[thick] (4,0)--(4,-.3);

\draw[thick] (-1,-2) to[out=-90,in=-90] (2,-1.75);  
\draw[thick] (-1,-2)--(-1,2);
 \draw[thick] (1,-.75) to[out=-90,in=-90] (4,-.75);

\node at (-3.5,0) {$\scalebox{1}{$\CL(T_+,T_-)=$}$};

\node at (0,-1.2) {$\;$};

\end{tikzpicture}
\]
If we start from an element $g$ of $F$, we consider $\iota(g)\in F_3$
and follow the procedure described above.
\begin{notation} We write $\CL(g)$ for the knot/link produced with the reduced tree diagram representing $g$.
 \end{notation} 
 Within this framework, an Alexander type theorem holds, that is for any unorientend link $L$ there exists an element $g$ in $F$ such that $\CL(g)=L$, \cite{Jo14}.
In passing, we mention that the construction in \cite{Jo14} is not "optimal", for the element $g$ tend to be "big" (in the sense that they have several leaves).
A more methodological approach to finding elements of $F$ that correspond to tangle sums or concatenations of (positive) $n$-crossing tangles
was recently pursued in \cite{GP}.

 \section{Thompson permutations}
 In this section we assign a permutation to each element of the Brown-Thompson group $F_3$ (and also to each element of $F$, seeing it as $\iota(F)\leq F_3$).

We briefly fix the notation for the permutations, \cite[Chapter 3]{Her}.
Given $k$ distinct integers $i_1, \ldots , i_k$ in $\{0, \ldots , n\}$, the symbol $(i_1, \ldots , i_k)$ represents the permutation $p: \{0, \ldots , n\}\to \{0, \ldots , n\}$, where 
$p(i_j)=i_{j+1}$ for $j<k$, $p(i_k)=i_1$, and $p(s)=s$ for all $s\in \{0, \ldots , n\}\setminus \{i_1, \ldots , i_k\}$. When $k=2$, the permutation is called a transposition. A permutation of the form $(i_1, \ldots , i_k)$ is called a $k$-cycle. Two cycles are said to be disjoint if they have no integers in common. Every permutation is the product of disjoint cycles.
 
First, given a ternary tree with $k$ leaves we are going to construct a permutation of $S_{k+1}$, acting on $\{0,1,\ldots , k\}$, which is the product of $(k+1)/2$ transpositions ($k$ is always odd, so set $k=2n+1$). 
Given a rooted ternary tree, say with $2n+1$ leaves, we define  a permutation on the set $\{0, \ldots , 2n+1\}$ associated with it.
We illustrate it with a couple of examples. Consider the pair of trees
\[
\begin{tikzpicture}[x=.5cm, y=.5cm,
    every edge/.style={
        draw,
      postaction={decorate,
                    decoration={markings}
                   }
        }
]

\node at (-3.5,0) {$\scalebox{1}{$T_+=$}$};

\draw[thick] (2,2)--(1,3)--(-2,0);

\draw[thick] (0,0) -- (2,2)--(4,0);
 \draw[thick] (1,1) -- (2,0);

\draw[thick] (0,0) -- (2,2)--(4,0);
\draw[thick] (1,1) -- (1,0);
\draw[thick] (1,1) -- (2,0);
\draw[thick] (2,2) -- (3,0);

\draw[thick] (1,3) -- (-1,0);

 \draw[thick] (1,3)--(1,3.5);

\node at (-2,-.75) {$\scalebox{.75}{$1$}$};
\node at (-1,-.75) {$\scalebox{.75}{$2$}$};
\node at (0,-.75) {$\scalebox{.75}{$3$}$};
\node at (1,-.75) {$\scalebox{.75}{$4$}$};
\node at (2,-.75) {$\scalebox{.75}{$5$}$};
\node at (3,-.75) {$\scalebox{.75}{$6$}$};
\node at (4,-.75) {$\scalebox{.75}{$7$}$};

\end{tikzpicture}
\qquad
\begin{tikzpicture}[x=.5cm, y=.5cm,
    every edge/.style={
        draw,
      postaction={decorate,
                    decoration={markings}
                   }
        }
]

\node at (-3.5,0) {$\scalebox{1}{$T_+=$}$};

\draw[thick] (2,2)--(1,3)--(-2,0);

\draw[thick] (0,0) -- (2,2)--(4,0);

\draw[thick] (0,0) -- (2,2);
\draw[thick] (2,2) -- (1,0);
\draw[thick] (3,1) -- (2,0);
\draw[thick] (3,1) -- (3,0);

\draw[thick] (1,3) -- (-1,0);

 \draw[thick] (1,3)--(1,3.5);

\node at (-2,-.75) {$\scalebox{.75}{$1$}$};
\node at (-1,-.75) {$\scalebox{.75}{$2$}$};
\node at (0,-.75) {$\scalebox{.75}{$3$}$};
\node at (1,-.75) {$\scalebox{.75}{$4$}$};
\node at (2,-.75) {$\scalebox{.75}{$5$}$};
\node at (3,-.75) {$\scalebox{.75}{$6$}$};
\node at (4,-.75) {$\scalebox{.75}{$7$}$};

\end{tikzpicture}
\]
where we numbered the leaves of each tree from left to right (starting from $1$).
We start with the   tree $T_+$.
We consider each leaf and take a path according to the rules displayed in Figure \ref{rules}. Each path ends when we meet another leaf or the root
(the paths for $T_+$ are highlighted in red in the figure below).   

\[
\begin{tikzpicture}[x=.5cm, y=.5cm,
    every edge/.style={
        draw,
      postaction={decorate,
                    decoration={markings}
                   }
        }
]

\draw[thick, red] (2,2)--(1,3)--(-2,0);

\draw[thick] (0,0) -- (2,2)--(4,0);
 \draw[thick] (1,1) -- (2,0);

\draw[thick] (0,0) -- (2,2)--(4,0);
\draw[thick] (1,1) -- (1,0);
\draw[thick] (1,1) -- (2,0);
\draw[thick, red] (2,2) -- (3,0);

\draw[thick] (1,3) -- (-1,0);

 \draw[thick] (1,3)--(1,3.5);

\node at (-2,-.75) {$\scalebox{.75}{$1$}$};
\node at (-1,-.75) {$\scalebox{.75}{$2$}$};
\node at (0,-.75) {$\scalebox{.75}{$3$}$};
\node at (1,-.75) {$\scalebox{.75}{$4$}$};
\node at (2,-.75) {$\scalebox{.75}{$5$}$};
\node at (3,-.75) {$\scalebox{.75}{$6$}$};
\node at (4,-.75) {$\scalebox{.75}{$7$}$};

\end{tikzpicture}
\; \;
\begin{tikzpicture}[x=.5cm, y=.5cm,
    every edge/.style={
        draw,
      postaction={decorate,
                    decoration={markings}
                   }
        }
]

\draw[thick] (2,2)--(1,3)--(-2,0);

\draw[thick] (0,0) -- (2,2)--(4,0);
 \draw[thick] (1,1) -- (2,0);

\draw[thick] (0,0) -- (2,2)--(4,0);
\draw[thick] (1,1) -- (1,0);
\draw[thick] (1,1) -- (2,0);
\draw[thick] (2,2) -- (3,0);

\draw[thick, red] (1,3) -- (-1,0);

 \draw[thick, red] (1,3)--(1,3.5);

\node at (-2,-.75) {$\scalebox{.75}{$1$}$};
\node at (-1,-.75) {$\scalebox{.75}{$2$}$};
\node at (0,-.75) {$\scalebox{.75}{$3$}$};
\node at (1,-.75) {$\scalebox{.75}{$4$}$};
\node at (2,-.75) {$\scalebox{.75}{$5$}$};
\node at (3,-.75) {$\scalebox{.75}{$6$}$};
\node at (4,-.75) {$\scalebox{.75}{$7$}$};
\end{tikzpicture}
\; \; 
\begin{tikzpicture}[x=.5cm, y=.5cm,
    every edge/.style={
        draw,
      postaction={decorate,
                    decoration={markings}
                   }
        }
]

\draw[thick] (2,2)--(1,3)--(-2,0);

\draw[thick] (1,1) -- (2,2)--(4,0);
 \draw[thick] (1,1) -- (2,0);

\draw[thick, red] (0,0) -- (1,1);
\draw[thick] (1,1) -- (1,0);
\draw[thick, red] (1,1) -- (2,0);
\draw[thick] (2,2) -- (3,0);

\draw[thick] (1,3) -- (-1,0);

 \draw[thick] (1,3)--(1,3.5);

\node at (-2,-.75) {$\scalebox{.75}{$1$}$};
\node at (-1,-.75) {$\scalebox{.75}{$2$}$};
\node at (0,-.75) {$\scalebox{.75}{$3$}$};
\node at (1,-.75) {$\scalebox{.75}{$4$}$};
\node at (2,-.75) {$\scalebox{.75}{$5$}$};
\node at (3,-.75) {$\scalebox{.75}{$6$}$};
\node at (4,-.75) {$\scalebox{.75}{$7$}$};

\end{tikzpicture}
\; \; 
\begin{tikzpicture}[x=.5cm, y=.5cm,
    every edge/.style={
        draw,
      postaction={decorate,
                    decoration={markings}
                   }
        }
]

\draw[thick] (2,2)--(1,3)--(-2,0);

\draw[thick, red] (1,1) -- (2,2)--(4,0);
 \draw[thick] (1,1) -- (2,0);

\draw[thick] (0,0) -- (1,1);
\draw[thick, red] (1,1) -- (1,0);
\draw[thick] (1,1) -- (2,0);
\draw[thick] (2,2) -- (3,0);

\draw[thick] (1,3) -- (-1,0);

 \draw[thick] (1,3)--(1,3.5);

\node at (-2,-.75) {$\scalebox{.75}{$1$}$};
\node at (-1,-.75) {$\scalebox{.75}{$2$}$};
\node at (0,-.75) {$\scalebox{.75}{$3$}$};
\node at (1,-.75) {$\scalebox{.75}{$4$}$};
\node at (2,-.75) {$\scalebox{.75}{$5$}$};
\node at (3,-.75) {$\scalebox{.75}{$6$}$};
\node at (4,-.75) {$\scalebox{.75}{$7$}$};

\end{tikzpicture}
\] 
 
  \begin{figure}[h]
 \[\begin{tikzpicture}[x=1cm, y=1cm,
    every edge/.style={
        draw,
      postaction={decorate,
                    decoration={markings}
                   }
        }
]

\draw[thick] (0.5,0.75)--(.5,0);
\draw[->, thick,red] (0,0)--(.25,.25);
\draw[->, thick,red] (.5,.5)--(.8,.2);
\draw[thick,red] (0,0)--(.5,.5)--(1,0);
 
\end{tikzpicture}
\quad
\begin{tikzpicture}[x=1cm, y=1cm,
    every edge/.style={
        draw,
      postaction={decorate,
                    decoration={markings}
                   }
        }
]

\draw[thick] (0.5,0.75)--(.5,0);
\draw[->, thick,red] (0.5,0.5)--(.2,.2);
\draw[->, thick,red] (1,0)--(.75,.25);
\draw[thick,red] (0,0)--(.5,.5)--(1,0);
 \end{tikzpicture}
\quad
\begin{tikzpicture}[x=1cm, y=1cm,
    every edge/.style={
        draw,
      postaction={decorate,
                    decoration={markings}
                   }
        }
]

\draw[->, red, thick] (0.5,0)--(.5,.25);
\draw[->, red, thick] (0.5,0.25)--(.5,.75);
\draw[thick] (0.5,0.5)--(.25,.25);
\draw[thick] (1,0)--(.75,.25);
\draw[thick] (0,0)--(.5,.5)--(1,0);
 \end{tikzpicture}
\quad
\begin{tikzpicture}[x=1cm, y=1cm,
    every edge/.style={
        draw,
      postaction={decorate,
                    decoration={markings}
                   }
        }
]

\draw[->, red, thick] (0.5,0.75)--(.5,.55);
\draw[->, red, thick] (0.5,0.75)--(.5,0);
\draw[thick] (0.5,0.5)--(.25,.25);
\draw[thick] (1,0)--(.75,.25);
\draw[thick] (0,0)--(.5,.5)--(1,0);
 \end{tikzpicture}
\]
\caption{Rules for defining the permutation.}  \label{rules}
\end{figure}
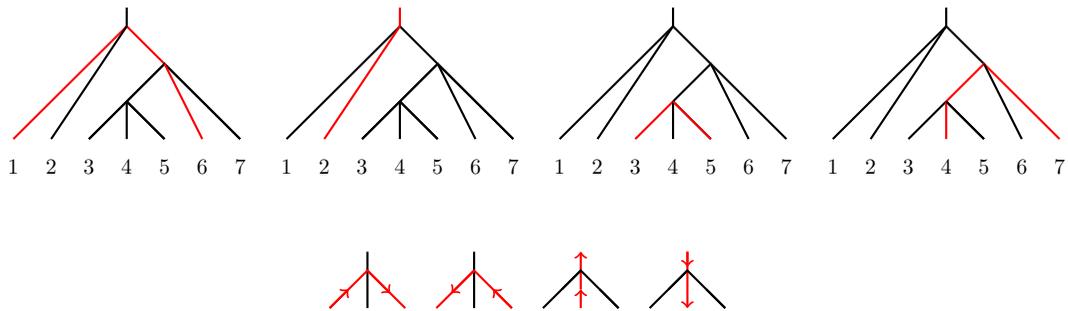
We note that in every tree there exists exactly one path from a leaf, say $f$, to the root. For this path we consider the 
  permutation $(0,f)$. For example, in our case we have $(1,6)$, $(2,0)$, $(3,5)$, $(4,7)$, $(5,3)$, $(6,1)$, $(7,4)$.    
Since all the transpositions (but the one corresponding to the root) occur exactly twice, we set aside only one of each.
Now we define the permutation $\pi(T_+): \{0,1, \ldots , 2n+1\}\to\{0,1, \ldots , 2n+1\}$ to be the product of all these transpositions. 
We call $\pi(T_+)$ the tangled permutation associated with $T_+$.
In this example, we get $\pi(T_+)=(1,6)(2,0)(3,5)(4,7)$.

For the second tree we follow the same procedure. 
For $T_-$ the paths are

\[
\begin{tikzpicture}[x=.5cm, y=.5cm,
    every edge/.style={
        draw,
      postaction={decorate,
                    decoration={markings}
                   }
        }
]

\draw[thick, red] (2,2)--(1,3)--(-2,0);

\draw[thick] (0,0) -- (2,2)--(4,0);

\draw[thick] (0,0) -- (2,2);
\draw[thick, red] (2,2) -- (1,0);
\draw[thick] (3,1) -- (2,0);
\draw[thick] (3,1) -- (3,0);

\draw[thick] (1,3) -- (-1,0);

 \draw[thick] (1,3)--(1,3.5);

\node at (-2,-.75) {$\scalebox{.75}{$1$}$};
\node at (-1,-.75) {$\scalebox{.75}{$2$}$};
\node at (0,-.75) {$\scalebox{.75}{$3$}$};
\node at (1,-.75) {$\scalebox{.75}{$4$}$};
\node at (2,-.75) {$\scalebox{.75}{$5$}$};
\node at (3,-.75) {$\scalebox{.75}{$6$}$};
\node at (4,-.75) {$\scalebox{.75}{$7$}$};

\end{tikzpicture}
\; \; 
\begin{tikzpicture}[x=.5cm, y=.5cm,
    every edge/.style={
        draw,
      postaction={decorate,
                    decoration={markings}
                   }
        }
]

\draw[thick] (2,2)--(1,3)--(-2,0);

\draw[thick] (0,0) -- (2,2)--(4,0);

\draw[thick] (0,0) -- (2,2);
\draw[thick] (2,2) -- (1,0);
\draw[thick] (3,1) -- (2,0);
\draw[thick] (3,1) -- (3,0);

\draw[thick, red] (1,3) -- (-1,0);

 \draw[thick, red] (1,3)--(1,3.5);

\node at (-2,-.75) {$\scalebox{.75}{$1$}$};
\node at (-1,-.75) {$\scalebox{.75}{$2$}$};
\node at (0,-.75) {$\scalebox{.75}{$3$}$};
\node at (1,-.75) {$\scalebox{.75}{$4$}$};
\node at (2,-.75) {$\scalebox{.75}{$5$}$};
\node at (3,-.75) {$\scalebox{.75}{$6$}$};
\node at (4,-.75) {$\scalebox{.75}{$7$}$};

\end{tikzpicture}
\;
\;
\begin{tikzpicture}[x=.5cm, y=.5cm,
    every edge/.style={
        draw,
      postaction={decorate,
                    decoration={markings}
                   }
        }
]

\draw[thick] (2,2)--(1,3)--(-2,0);

\draw[thick, red] (0,0) -- (2,2);
\draw[thick, red] (2,2)--(3,1);
\draw[thick] (3,1)--(4,0);

 \draw[thick] (2,2) -- (1,0);
\draw[thick] (3,1) -- (2,0);
\draw[thick, red] (3,1) -- (3,0);

\draw[thick] (1,3) -- (-1,0);

 \draw[thick] (1,3)--(1,3.5);

\node at (-2,-.75) {$\scalebox{.75}{$1$}$};
\node at (-1,-.75) {$\scalebox{.75}{$2$}$};
\node at (0,-.75) {$\scalebox{.75}{$3$}$};
\node at (1,-.75) {$\scalebox{.75}{$4$}$};
\node at (2,-.75) {$\scalebox{.75}{$5$}$};
\node at (3,-.75) {$\scalebox{.75}{$6$}$};
\node at (4,-.75) {$\scalebox{.75}{$7$}$};

\end{tikzpicture} \; \;
\begin{tikzpicture}[x=.5cm, y=.5cm,
    every edge/.style={
        draw,
      postaction={decorate,
                    decoration={markings}
                   }
        }
]

\draw[thick] (2,2)--(1,3)--(-2,0);

\draw[thick] (0,0) -- (2,2)--(3,1);
\draw[thick, red] (3,1)--(4,0);

\draw[thick] (0,0) -- (2,2);
\draw[thick] (2,2) -- (1,0);
\draw[thick, red] (3,1) -- (2,0);
\draw[thick] (3,1) -- (3,0);

\draw[thick] (1,3) -- (-1,0);

 \draw[thick] (1,3)--(1,3.5);

\node at (-2,-.75) {$\scalebox{.75}{$1$}$};
\node at (-1,-.75) {$\scalebox{.75}{$2$}$};
\node at (0,-.75) {$\scalebox{.75}{$3$}$};
\node at (1,-.75) {$\scalebox{.75}{$4$}$};
\node at (2,-.75) {$\scalebox{.75}{$5$}$};
\node at (3,-.75) {$\scalebox{.75}{$6$}$};
\node at (4,-.75) {$\scalebox{.75}{$7$}$};

\end{tikzpicture}
\]
and the transpositions are $(1,4)$, $(2,0)$, $(3,6)$, $(4,1)$, $(5,7)$, $(6,3)$, $(7,5)$. Therefore, the permutation associated to $T_-$ is $\pi(T_-)=(0,2)(1,4)(3,6)(5,7)$. 

Finally, given the ternary tree diagram  $(T_+,T_-)$ we associate a   permutation $\CP(T_+,T_-)$ in the following way.
First, we draw $\mathcal{L}(T_+,T_-)$.  The link diagram is divided into two parts: one in upper-half plane, one in the lower-half plane.
Then we draw the $x$-axis and number the intersection points. 
Say from $0$ to $n$. Our permutation is going to act on $\{0, 1, \ldots ,  n\}$.
Now for each component, we orient upward the strand passing through its leftmost intersection point with the $x$-axis.
This gives an orientation to the whole link.
For each $i$ in ${1, ..., n}$, we decree that the Thompson permutation $\mathcal{P}(g)$ maps $i$ to the next integer (on the $x$-axis) met in $\mathcal{L}(g)$ moving along the given direction.

An equivalent description makes use of the permutations  $\pi(T_+)$ and $\pi(T_-)$
defined before. This is what we use for our numerical explorations of the positive elements.

Take $0$ and  alternately apply  $\pi(T_+)$ and $\pi(T_-)$ to it  until you obtain $0$, that is consider 
$0$, $\pi(T_+)(0)$, $\pi(T_-)(\pi(T_+)(0))$, $\pi(T_+)(\pi(T_-)(\pi(T_+)(0)))$, \ldots 
We set aside the cycle $(0, \pi(T_+)(0), \ldots, \ldots , \pi(T_-)(\pi(T_+)\pi(T_-))^{k-1}(0))$.
Repeat this procedure with the smallest non-negative integer not contained in this cycle and get a second cycle. 
Continue in the same manner until you finish.
 The Thompson permutation $\CP(T_+,T_-)$ is the product of these cycles.
  In our example we have 
 $
\CP(T_+,T_-)=
(1,6,3,5,7,4) (0,2)$.
\[
\begin{tikzpicture}[x=.35cm, y=.35cm,
    every edge/.style={
        draw,
      postaction={decorate,
                    decoration={markings}
                   }
        }
]

\node at (-5.5,0) {$\scalebox{1}{$\CL(x_2)=$}$};
\node at (-1.25,-3.25) {\;};

\draw[thick] (-2,0) to[out=90,in=90] (2,1.7);  
\draw[thick] (0,0) to[out=90,in=90] (2,0);  
\draw[thick] (1,.8) to[out=90,in=90] (4,0);  

\draw[thick] (-2,0) to[out=-90,in=-90] (2,-1.7);  
\draw[thick] (2,0) to[out=-90,in=-90] (4,0);  
\draw[thick] (0,0) to[out=-90,in=-90] (3,-.8);  

 \draw[thick] (-3,2.7)--(-3,-2.7);
\draw[thick] (-3,2.7) to[out=90,in=90] (1,2.7);  
\draw[thick] (-3,-2.7) to[out=-90,in=-90] (1,-2.7);  

\draw[thick] (1,2.1) to[out=-90,in=90] (-1,0);
\draw[thick] (-1,0) to[out=-90,in=135] (1,-2.1);
\draw[thick] (1,.4) to[out=-90,in=90] (2,-1);
\draw[thick] (2,1.2) to[out=-90,in=90] (3,-.3);

\end{tikzpicture}\; \; 
\begin{tikzpicture}[x=.35cm, y=.35cm,
    every edge/.style={
        draw,
      postaction={decorate,
                    decoration={markings}
                   }
        }
]

\draw[thick, red, dashed] (-4,0)--(5,0);

\node[red] at (-3.5,-.75) {$\scalebox{.75}{$0$}$};
\node[red] at (-2.2,-.75) {$\scalebox{.75}{$1$}$};
\node[red] at (-1.2,-.75) {$\scalebox{.75}{$2$}$};
\node[red] at (0,-.75) {$\scalebox{.75}{$3$}$};
\node[red] at (1,-.75) {$\scalebox{.75}{$4$}$};
\node[red] at (2.2,-.75) {$\scalebox{.75}{$5$}$};
\node[red] at (3.2,.5) {$\scalebox{.75}{$6$}$};
\node[red] at (4,-.75) {$\scalebox{.75}{$7$}$};

\node at (-5,0) {$\scalebox{1}{$ =$}$};
\node at (-1.25,-3.25) {\;};

\draw[thick, ->] (-2,0) to[out=90,in=90] (2,1.7);  
\draw[thick] (0,0) to[out=90,in=90] (2,0);  
\draw[thick] (1,.8) to[out=90,in=90] (4,0);  

\draw[thick] (-2,0) to[out=-90,in=-90] (2,-1.7);  
\draw[thick] (2,0) to[out=-90,in=-90] (4,0);  
\draw[thick] (0,0) to[out=-90,in=-90] (3,-.8);  

 \draw[thick] (-3,2.7)--(-3,-2.7);
\draw[thick, ->] (-3,2.7) to[out=90,in=90] (1,2.7);  
\draw[thick] (-3,-2.7) to[out=-90,in=-90] (1,-2.7);  

\draw[thick] (1,2.1) to[out=-90,in=90] (-1,0);
\draw[thick] (-1,0) to[out=-90,in=135] (1,-2.1);
\draw[thick] (1,.4) to[out=-90,in=90] (2,-1);
\draw[thick] (2,1.2) to[out=-90,in=90] (3,-.3);

\end{tikzpicture}
\]
When we start with a binary tree diagram $(T_+,T_-)$, that is an element of $F$, we consider $\iota(T_+,T_-)$ and its Thompson permutation is $\CP(\iota(T_+,T_-))$.

\begin{notation}
By the above discussion, we may assign an orientation to any link produced out of elements of $F_3$: we start looking each component of the link and  orient upward the left-most arc  of the link passing through the $x$-axis. For $g$ in $F$ or $F_3$, we denote by $\vec{{\text \textsterling}}(g)$ the corresponding oriented link. 
We recall the notation existing in the literature on Jones' construction of links:
\begin{itemize}
\item for $g$ in $F$ or $F_3$, $\CL(g)$ is the corresponding unoriented link:
\item for $g$ in $\vec{F}$ or $\vec{F}_3$, $\vec{\CL}(g)$ is the associated oriented link.
\end{itemize}
\end{notation}
A natural question is whether with this procedure we get all oriented links. We provide a positive answer with the following Alexander type theorem. The situation is thus analogous to that of $F$ and $F_3$, \cite{Jo14, A}.
\begin{theorem}\label{theo1}
For every oriented link $\vec{L}$, there exists an element $g$ in $F$ such that  $\vec{{\text \textsterling}}(g)$.
\end{theorem}
\begin{proof}
Given an oriented link $\vec{L}$, first we ignore its orientation and consider the underlying unoriented link $L$.
We then use Jones' method from \cite{Jo14} to produce an element $g$ of $F$
such that $\CL(g)$ is $L$.
Now if we orient $\CL(g)$ according to  the convention of this paper, the oriented link that we obtain may or may not be equal to $\vec{L}$.
If the orientation of a component of the link is not the same, 
consider the pair of glued leaves whose corresponding arc in the link diagram has opposite orientation and 
replace them with a copy of $y_0$. 
This reverses the orientation of the component.
We can use the same method to change the orientation of a component in a link $\CL(g)$, with $g\in F_3$. In this case use $x_0$ instead of $y_0$.
Below follow the transformations for  $F$ and $F_3$, 
and on drawing to show what happens on the level of link diagrams.
\[
\begin{tikzpicture}[x=.35cm, y=.35cm,
    every edge/.style={
        draw,
      postaction={decorate,
                    decoration={markings}
                   }
        }
]

\node at (-1.25,-4) {\;};

 \draw[thick] (2,2.5)--(2,-2.5);
 
   \fill (2,0)  circle[radius=1.5pt];

\end{tikzpicture}
\quad
\begin{tikzpicture}[x=.35cm, y=.35cm,
    every edge/.style={
        draw,
      postaction={decorate,
                    decoration={markings}
                   }
        }
]

   \fill (0,0)  circle[radius=1.5pt];
   \fill (2,0)  circle[radius=1.5pt];
   \fill (4,0)  circle[radius=1.5pt];

\node at (-1.5,0) {$\scalebox{1}{$\mapsto$}$};
\node at (-1.25,-4) {\;};

\draw[thick] (0,0) -- (2,2)--(4,0)--(2,-2)--(0,0);
 \draw[thick] (1,1) -- (2,0)--(3,-1);

 \draw[thick] (2,2)--(2,2.5);

 \draw[thick] (2,-2)--(2,-2.5);

\end{tikzpicture}
\qquad
\begin{tikzpicture}[x=.35cm, y=.35cm,
    every edge/.style={
        draw,
      postaction={decorate,
                    decoration={markings}
                   }
        }
]

\node at (-1.25,-4) {\;};

 \draw[thick] (2,2.5)--(2,-2.5);
 
   \fill (2,0)  circle[radius=1.5pt];

\end{tikzpicture}
\quad
\begin{tikzpicture}[x=.35cm, y=.35cm,
    every edge/.style={
        draw,
      postaction={decorate,
                    decoration={markings}
                   }
        }
]

   \fill (0,0)  circle[radius=1.5pt];
   \fill (1,0)  circle[radius=1.5pt];
   \fill (2,0)  circle[radius=1.5pt];
   \fill (3,0)  circle[radius=1.5pt];
   \fill (4,0)  circle[radius=1.5pt];

\node at (-1.5,0) {$\scalebox{1}{$\mapsto$}$};
\node at (-1.25,-4) {\;};

\draw[thick] (0,0) -- (2,2)--(4,0)--(2,-2)--(0,0);
 \draw[thick] (1,1) -- (2,0)--(3,-1);

 \draw[thick] (2,2)--(2,2.5);

 \draw[thick] (2,-2)--(2,-2.5);

 \draw[thick] (1,1)--(1,0)--(2,-2);
 \draw[thick] (2,2)--(3,0)--(3,-1);

\end{tikzpicture}
\qquad
\qquad
\begin{tikzpicture}[x=.35cm, y=.35cm,
    every edge/.style={
        draw,
      postaction={decorate,
                    decoration={markings}
                   }
        }
]

 \draw[thick, <-] (2,2.5)--(2,-3);
 \node at (-1.5,-4) {$\scalebox{.75}{$\,$}$};

\end{tikzpicture}
\begin{tikzpicture}[x=.35cm, y=.35cm,
    every edge/.style={
        draw,
      postaction={decorate,
                    decoration={markings}
                   }
        }
]

\node at (-1.5,0) {$\scalebox{.75}{$\mapsto$}$};
\node at (-1.5,-4) {$\scalebox{.75}{$\,$}$};

\draw[thick, ->] (2,2.5) to[out=90,in=90] (2,1.7);  
\draw[thick, ->] (0,0) to[out=90,in=90] (2,0);  
\draw[thick] (1,.8) to[out=90,in=90] (4,0);  

\draw[thick,>-] (2,-2.5) to[out=-90,in=-90] (2,-1.7);  
\draw[thick] (2,0) to[out=-90,in=-90] (4,0);  
\draw[thick] (0,0) to[out=-90,in=-90] (3,-.8);

\draw[thick] (1,.4) to[out=-90,in=90] (2,-1);
\draw[thick] (2,1.2) to[out=-90,in=90] (3,-.3);

\end{tikzpicture}
\]

\end{proof}
 
\begin{remark}
Let $g$ be in $\vec{F}$ or $\vec{F}_3$. If $\vec{\CL}(g)$ is knot (i.e. it has only one connected component), then 
$\vec{\CL}(g)$ 
is equal to $\vec{{\text \textsterling}}(g)$. 
For links this is in general not true.
\end{remark}

\section{Computational explorations in $F_{3,+}$}
  
A positive element of $F_{3,+}$ can be written uniquely as 
$x_0^{a_0}\cdots x_n^{a_n}$ for some $n$, $a_0$, $\ldots$, $a_{n-1}\in \IN_0=\{0, 1, 2, \ldots\}$ and 
$a_n\in \IN=\{1, 2, \ldots\}$.
We call $n$ the width and $\max_{i}a_i$ the height of the element.

In this section we present the results of some numerical experiments where we fix width and vary the height.
For each pair width and height we group the permutations by the number of cycles in the cycle decomposition (or orbits).
We examined 1350210 elements in total.
Based on our explorations, we present some conjectures.

In the following, $w$ is the width, $h$ is the height. Note that there are $(h+1)^{w}$ permutation with width at most $w$ and height at most $h$.
It is easy to show that $\mathcal{L}(y_0^n)$ is the unknot for all $n\geq 0$, so we only consider the case $w\geq 2$.
Here follow five subsections for the values $w=2, 3, 4, 5, 6$. 
The results are also presented in Figures \ref{tablesexplorations}
and \ref{permutationsdistribution}.

\begin{definition}
For each pair $(w,h)$, 
we group the corresponding permutations with respect to the number of orbits, we call these groups \emph{classes}. 
\end{definition}

The data is available in \cite{AiIo} and the code used to produce it is in \url{https://github.com/valerianoaiello/Positive-Thompson-knots}.

\begin{conjecture} 
\begin{enumerate}
\item
For all $w$ and $h$, let $M$ be the maximum number of cycles of a permutation with width at most $w$ and height at most $h$. Then, 
for any integer $j$ in $\{1, \ldots , M\}$
there exists a permutation whose weight and height are at most $w$ and $h$, respectively, whose number of orbits is $j$.
\item 
For $w=2$ we considered all permutations up to height equal to $100$  (the number of permutation of height at most $100$ is $10201$).
For $w=2$ and $h\geq 0$, the maximum number of cycles is $h+1$.
For $h\geq 0$, the largest class  consists of the permutations with $1$ orbit.
\item
For $w=3$ we considered all permutations up to height equal to $35$  (the number of permutation of height at most $35$ is $42875$ permutations).
For $w=3$ and $h\geq 0$, the maximum number of cycles is $h+1$.
For $h\geq 3$, the largest class  consists of the permutations with $2$ distinct cycles.
\item
For $w=4$ and $h\geq 1$, the maximum number of cycles is $2h$.
\item
For $w=5$ and $h\geq 2$, the maximum number of cycles is $2h$
and the largest class  consists of the permutations with $h-1$ distinct cycles.
\item
For $w=6$ and $h\geq 2$, the maximum number of cycles is $3h-1$.
\item
For $w=7$ and $h\geq 2$, the maximum number of cycles is $3h-1$.
For $h\geq 2$, the largest class consists of the permutations with $h$ distinct cycles.
\end{enumerate}
\end{conjecture}

\begin{figure}
{\footnotesize
\[
\begin{array}{|c|c|c|c|}
\hline w= 2   & \text{max number of orbits} & \text{largest class} &  \text{number of permutations}  \\
\hline
h =0 & 1 & 1 & 1\\
h =1 & 1 & 1 & 4\\
h \in [2, 100] & h+1 & 1 & \text{up to } 10201\\
\hline &  & & \\
%
\hline w=3   & \text{max number of orbits} & \text{largest class} &  \text{number of permutations}  \\
\hline
h =0 & 1 & 1 & 1\\
h =1 & 2 & 1 & 8\\
h =2 & 3 & 1 & 27\\
h \in  [3, 35] & h+1 & 2 & \text{up to } 42875 \\
\hline &  & & \\
\hline w=4   & \text{max number of orbits} & \text{largest class} &  \text{number of permutations}  \\
\hline
h =0 & 1 & 1 & 1\\
h =1 & 2 & 1, 2 & 16\\
h =2 & 4 & 2 & 81\\
h =3 & 6 & 2 & 256\\
h =4 & 8 & 3 & 625\\
h =5 & 10 & 5 & 1296\\
h =6 & 12 & 6 & 2401\\
h =7 & 14 & 6 & 4096\\
h =8 & 16 & 8 & 6561\\
h =9 & 18 & 9 & 10000\\
h =10 & 20 & 10 & 14641\\
h =11 & 22 & 11 & 20736\\
h =12 & 24 & 12 & 28561\\
\hline &  & & \\
\hline w=5   & \text{max number of orbits} & \text{largest class} &  \text{number of permutations}  \\
\hline
h =0 & 1 & 1 & 1\\
h =1 & 3 & 1 & 32\\
h \in [2, 11] & 2h & h-1 &  \text{up to } 248832\\ 
\hline &  & & \\
\hline w=6   & \text{max number of orbits} & \text{largest class} &  \text{number of permutations}  \\
\hline
h =0 & 1 & 1 & 1\\
h =1 & 3 & 1 & 64\\
h =2 & 5 & 2 & 729\\
h =3 & 8 & 3 & 4096\\
h =4 & 11 & 4 & 15625\\
h =5 & 14 & 5 & 46656\\
h=6 & 17 &  6 & 117649\\
h=7 & 20 & 7 & 262144\\
h=8 & 23 & 9 & 531441\\
\hline &  & & \\
\hline w=7   & \text{max number of orbits} & \text{largest class} &  \text{number of permutations}  \\
\hline
h =0 & 1 & 1 & 1\\
h =1 & 4 & 2 & 128\\
h\in [2, 6] & 3h-1 & h &   \text{up to } 823543 \\
\hline
\end{array}
\]
}
\caption{This table contains the data on permutations of width 2, 3, 4, 5, 6, 7.}\label{tablesexplorations}
\end{figure}

 \begin{figure}[h]
$$
\includegraphics[scale=0.5]{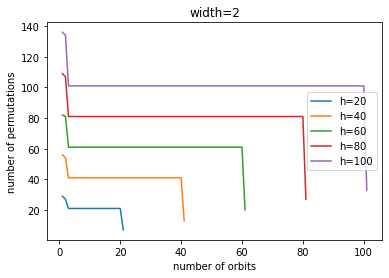}
\includegraphics[scale=0.5]{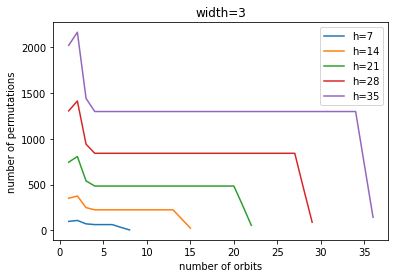}
$$
$$
\includegraphics[scale=0.5]{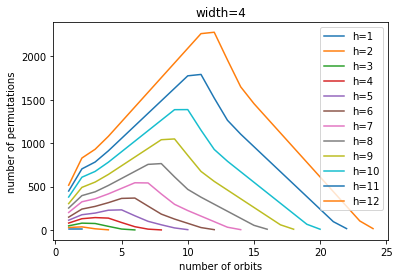}
\includegraphics[scale=0.5]{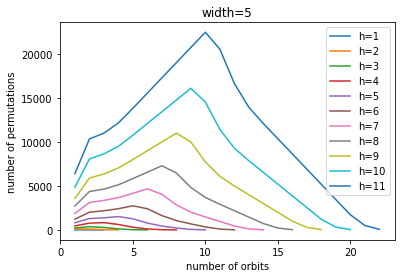}
$$
$$
\includegraphics[scale=0.5]{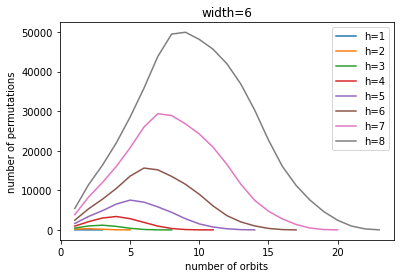}
\includegraphics[scale=0.5]{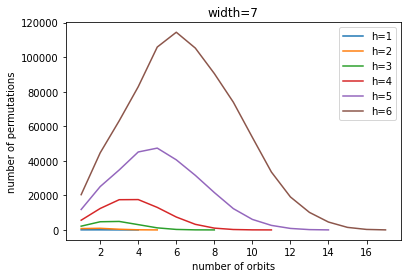}
 $$ 
$$
\includegraphics[scale=0.5]{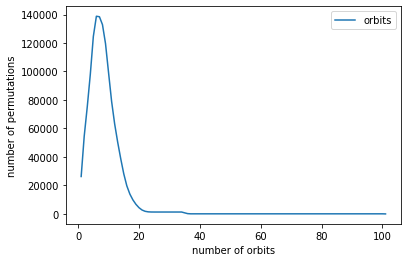}
 $$ 
 \caption{The distributions of permutations for width equal to $2, 3, 4, 5, 6, 7$ and the one cumulative with all the permutations considered in this article. The largest class of permutations in the whole dataset is the one with $6$ cycles.}\label{permutationsdistribution}
   \end{figure}
 
\section*{Appendix}
 In this appendix we briefly present the code developed for this project.
The main files are \emph{monoid\_elements\_generator.py} and \emph{positive\_bt\_permutations.py}.

In \emph{monoid\_elements\_generator.py}  we introduce the class 
\emph{MonoidElementsGenerator}, which contains two methods: 
\begin{itemize}
\item \emph{generate\_monoid\_elements},
\item \emph{random\_generate\_monoid\_elements}.
\end{itemize}
The former produces all positive elements of $F_3$ with height at most $h$ and width at most $w$.
The latter takes a positive integer $n$ as an argument and produces $n$ random positive elements with height at most $h$
and width at most $w$.

In  \emph{positive\_bt\_permutations.py} we present a function called \emph{whole\_permutation} that produces the Thompson permutations for positive elements of $F_3$.
It receives a natural number $k$ and a numpy array $v=(a_1, \ldots , a_w)$, 
whose components are all non-negative,
$w$ is the maximal width, 
$k$ is the number of leaves in the minimal representative in terms of ternary tree diagrams, 
while $v$ represents the exponents appearing
 in the description the element as $x_0^{a_1}\cdots x_{w-1}^{a_w}$.
The permutation produced by this function is given in the form of a list containing sublists. Each sublist is a cycle in the cycle notation of the permutation.
This function makes use of two other functions: \emph{bottom\_permutation} and \emph{top\_permutation}.
The former function computes the permutation associated with the bottom ternary tree of a positive Brown-Thompson element.
    The only input is a natural number, \emph{number\_of\_leaves}, which is an odd number.
    We only need an argument in this function because the bottom tree of a positive element has always the same shape.
The latter function  computes the permutation associated with the (ternary) top tree of a positive Brown-Thompson element.
    The inputs are \emph{monoid\_element}$=v$,
     and a natural number, \emph{number\_of\_leaves}, which is an odd number.
    The output of this function is a permutation.


\begin{thebibliography}{99}

\bibitem{A} V. Aiello,  \emph{On the Alexander Theorem for the oriented Thompson group $\vec{F}$}. Algebraic \& Geometric Topology 20 (2020) 429-438, preprint arXiv:1811.08323 (2018).

\bibitem{A2} V. Aiello, \emph{An introduction to Thompson knot theory and to Jones subgroups}.  accepted for publication in J. Knot Theory Ramifications, preprint arXiv:2211.15461 (2022).

\bibitem{ABC} V. Aiello, A. Brothier, R. Conti. \emph{Jones representations of Thompson's group $F$ arising from Temperley-Lieb-Jones algebras}, Int. Math. Res. Not. 15 (2021), 11209-11245, preprint arXiv:1901.10597 (2019).

\bibitem{AB} V. Aiello, S. Baader, \emph{Positive oriented Thompson links},  accepted for publication in Communications in Analysis and Geometry, preprint arXiv:2101.04534

\bibitem{AB2} V. Aiello, S. Baader, Arborescence of positive Thompson links, Pacific Journal of Mathematics 316 (2) (2022), 237-248, doi: 10.2140/pjm.2022.316.237 preprint arXiv:2106.13648

\bibitem{AiCo1} V. Aiello, R. Conti, \emph{Graph polynomials and link invariants as positive type functions on Thompson's group $F$}, 
 J. Knot Theory Ramifications 28 (2019), no. 2, 1950006, 17 pp.
doi: 10.1142/S0218216519500068 , preprint arXiv:1510.04428

\bibitem{AiCo2} V. Aiello, R. Conti, \emph{The Jones polynomial and functions of positive type on the oriented Jones-Thompson groups $\vec{F}$ and $\vec{T}$},  Complex Anal. Oper. Theory (2019) 13: 3127. doi: 10.1007/s11785-018-0866-6 ,  preprint arXiv:1603.03946 

\bibitem{ACJ} V. Aiello, R. Conti, V.F.R. Jones, \emph{The Homflypt polynomial and the oriented Thompson group}, Quantum Topol. 9 (2018), 461-472. preprint arXiv:1609.02484

\bibitem{AiIo} V. Aiello, S. Iovieno, \emph{Positive Thompson Permutations of width 7}, doi: 10.5281/zenodo.7424881, \url{https://zenodo.org/record/7424881}

\bibitem{AJ} V. Aiello, V.F.R. Jones, \emph{On spectral measures for certain unitary representations of R. Thompson's group F}. J. Funct. Anal., Volume 280, Issue 1, 1 January 2021, 108777, preprint arXiv:1905.05806 (2019).

\bibitem{TV} V. Aiello, T. Nagnibeda, \emph{On the oriented Thompson subgroup $\vec{F}_3$ and its relatives in higher Brown-Thompson groups}, Journal of Algebra and its Applications, Volume 21, Issue 07, 2250139,

\bibitem{TV2} V. Aiello, T. Nagnibeda, \emph{On the 3-colorable subgroup and maximal subgroups of Thompson's group F}, accepted for publication in the Annales l'Institut Fourier, preprint arXiv:2103.07885

\bibitem{TV3} V. Aiello, T. Nagnibeda, \emph{The planar $3$-colorable subgroup $\mathcal{E}$ of Thompson's group $F$ and its even part}, preprint.


\bibitem{B} J. Belk, Thompson's group F. Ph.D. Thesis (Cornell University).  preprint arXiv:0708.3609 (2007).


\bibitem{BM} J. Belk, F. Matucci. \emph{Conjugacy and dynamics in Thompson's groups.} Geometriae Dedicata 169 (2014): 239-261.

\bibitem{BJ} A. Brothier, V.F.R. Jones. \emph{Pythagorean representations of Thompson's groups}. Journal of Functional Analysis, Vol. 277 (2019) 2442-2469.
\bibitem{Bro} A. Brothier, D. Wijesena. \emph{Jones' representations of R. Thompson's groups not induced by finite-dimensional ones}. preprint arXiv:2211.08555 (2022).


\bibitem{Brown} K. S. Brown, \emph{Finiteness properties of groups}. Proceedings of the Northwestern conference on cohomology of groups (Evanston, Ill., 1985). J. Pure Appl. Algebra 44 (1987), no. 1--3, 45--75.

 
\bibitem{Cuntz} J. Cuntz,  \emph{Simple C$^*$-algebra generated by isometries.} Communications in mathematical physics 57.2 (1977): 173-185.

\bibitem{CFP}
J.W. Cannon, W.J Floyd,   W.R. Parry, 
Introductory notes on Richard Thompson's groups.
{\em L'Enseignement  Math\'ematique} 
{\bf 42} (1996): 215--256

 %
\bibitem{GS2} G. Golan, M. Sapir, \emph{On subgroups of R. Thompson's group $F$}, Transactions of the American Mathematical Society 369.12 (2017): 8857--8878%

\bibitem{GP} A. Grymski, E. Peters. \emph{Conway Rational Tangles and the Thompson Group.} preprint arXiv:2212.00100 (2022).


\bibitem{GS} V. Guba, M. Sapir. \emph{Diagram groups}. Vol. 620. American Mathematical Soc., 1997.

 
\bibitem{Her}  I.N. Herstein, Abstract Algebra, 1996,
Wiley.
 

\bibitem{jo2}V.F.R. Jones, Planar Algebras I, 
New Zealand Journal of Mathematics 52 (2021): 1-107.
preprint.
	math/9909027

\bibitem{Jo14} V.F.R. Jones, \emph{Some unitary representations of Thompson's groups $F$ and $T$}. J. Comb. Algebra {\bf 1} (2017), 1--44.

\bibitem{Jo16}V.F.R. Jones, \emph{A no-go theorem for the continuum  limit of a quantum spin chain}. \emph{Comm. Math. Phys.} {\bf 357} (2018), 295--317.

\bibitem{Jo18} V.F.R. Jones, \emph{On the construction of knots and links from Thompson's groups}.  In: Adams C. et al. (eds) Knots, Low-Dimensional Topology and Applications. KNOTS16 2016. Springer Proceedings in Mathematics \& Statistics, vol 284. Springer, Cham.

\bibitem{Jo19} V.F.R. Jones, \emph{Irreducibility of the Wysiwyg representations of Thompson's groups}, Representation Theory, Mathematical Physics, and Integrable Systems. Birkh\"auser, Cham, 2021. 411-430. preprint arXiv:1906.09619 (2019).
 

\end{thebibliography}
\end{document}